\documentclass{birkjour}
 \usepackage{graphicx,color,epsfig}
\usepackage{hyperref}
    \numberwithin{equation}{section}
\newtheorem{theorem}{Theorem}[section]

\newtheorem{example}[theorem]{Example}

\newtheorem{prop}[theorem]{Proposition}

\newtheorem{remark}[theorem]{Remark}

\begin{document}

\title[Symmetric functions and symmetrized bidisk]{Symmetric Schur-class functions on the bidisk and Schur-class functions on the symmetrized bidisk}
\author[R. Baran]{Radomi\l \ Baran}\address{Department of Mathematics\\ Jagiellonian University\\ 
Krak\'{o}w, Poland.
}\email{radomil.baran@doctoral.uj.edu.pl}

\author[H. J. Woerdeman]{Hugo J. Woerdeman}\address{Department of Mathematics\\ Drexel University\\ 
3141 Chestnut Street\\ Philadelphia, PA 19104, USA. } \email{hugo@math.drexel.edu} \thanks{The research of HJW is supported by National Science Foundation grant DMS 2348720.
}

\date{ }

\dedicatory{In Memory of Rien Kaashoek}

\begin{abstract}
We present some thoughts on the relation between  symmetric Schur-class functions on the bidisk and Schur-class functions on the symmetrized bidisk. Among other things, use of this relation leads to a finite dimensional realization result for rational matrix functions in the Schur-class on the symmetrized bidisk and also to a determinantal representation result for polynomials without zeros on the symmetrized bidisk. 
\end{abstract}

%\linenumbers


          \maketitle

\section{Introduction}

The \emph{symmetrized bidisk} is the domain
\[
\mathbb{G} = \left\{(s, p) \in \mathbb{C}^2 : \exists\, z, \zeta \in \mathbb{D},\ s = z + \zeta,\ p = z \zeta \right\},
\]
the image of the bidisk ${\mathbb D}^2 = \{ (z,\zeta) \in {\mathbb C}^2 : |z|<1, |\zeta|<1 \}$ under the symmetrization map $(z, \zeta) \mapsto (z + \zeta, z \zeta)$. Despite its non-convexity and lack of product structure, $\mathbb{G}$ supports a rich function theory. Of particular interest is the Schur class on $\mathbb{G}$, the set of analytic functions bounded by $1$, which has been extensively studied using operator-theoretic techniques.

A fundamental result of Agler and Young~\cite{MR3641771} shows that any Schur-class function on $\mathbb{G}$ admits a realization as a modified transfer function of a unitary operator colligation. By different methods Bhattacharyya and Sau \cite{MR3724147} also arrive at this realization result. These realizations reflect the geometry of $\mathbb{G}$ and generalize classical linear-fractional models from one variable. The realization theory is a useful tool in interpolation: the Nevanlinna-Pick problem on $\mathbb{G}$ admits a solution in terms of kernel positivity, similar in spirit to the classical theory, but requiring tools adapted to the symmetrized structure; see  \cite{MR3641771, MR3724147, MR4900535}.

These results are also connected to the study of $\Gamma$-contractions, pairs of commuting operators on Hilbert
space for which the symmetrized bidisk
is a spectral set, and have implications in control theory and multivariable operator theory. For further details, see~\cite{AglerYoung2004, BHATTACHARYYA2012577}.

In this paper we arrive at realizations for functions in the Schur class of the symmetrized bidisk ${\mathbb G}$ by first developing a realization result for symmetric functions (i.e., $f(z,\zeta)=f(\zeta,z)$) in the Schur class on the bidisk ${\mathbb D}^2$. Our approach has the advantage that we also establish a finite dimensional realization for rational Schur class functions on the symmetrized bidisk. In the case that the rational function is inner, the finite dimensional realization was established in \cite[Theorem 3.4]{MR4900535}. In addition, we explore Nevanlinna-Pick interpolation and determinantal representations both for symmetric functions on ${\mathbb D}^2$ (see Section 2) as well as functions on ${\mathbb G}$ (see Section 3).

\section{Symmetric functions}

We say that a function $f(z,\zeta)$ is {\em symmetric} if for all $(z,\zeta)$ in its domain, we have that $(\zeta, z)$ belongs to its domain and $f(z,\zeta)=f(\zeta,z)$. 

We denote Hilbert spaces by curly capitals, such as ${\mathcal H}, {\mathcal U}, {\mathcal Y}$, and we let ${\mathcal B}({\mathcal U}, {\mathcal Y})$ denote the Banach space of bounded linear operators acting ${\mathcal U} \to {\mathcal Y}$.
We let ${\mathcal S}_{{\mathbb D}^2}({\mathcal U}, {\mathcal Y})$ denote the Schur class on ${\mathbb D}^2$ of ${\mathcal B}({\mathcal U}, {\mathcal Y})$-valued functions; i.e., $f:{\mathbb D}^2 \to {\mathcal B}({\mathcal U}, {\mathcal Y})$ is analytic and its supremum norm $\| f \|_\infty$ is less than or equal to 1.

\begin{theorem}\label{thm1}
    We have that $f(z,\zeta)\in {\mathcal S}_{{\mathbb D}^2}({\mathcal U}, {\mathcal Y})$ is symmetric if and only if there exists a contractive colligation matrix
    \begin{equation}\label{coll} M=\begin{pmatrix} A_1 & A_2 & B\cr A_2 & A_1 & B \cr C & C & D \end{pmatrix} : \begin{matrix} {\mathcal H} \cr \oplus  \cr {\mathcal H} \cr \oplus\cr {\mathcal U}\end{matrix} \to \begin{matrix} {\mathcal H} \cr \oplus \cr  {\mathcal H}\cr \oplus \cr {\mathcal Y}\end{matrix}\end{equation} so that
    \begin{equation}\label{ff} f(z,\zeta)= D + \begin{pmatrix} C & C \end{pmatrix} \begin{pmatrix} zI & 0 \cr 0 & \zeta I \end{pmatrix} \left( I- \begin{pmatrix} A_1 & A_2 \cr A_2 & A_1 \end{pmatrix} \begin{pmatrix} zI & 0 \cr 0 & \zeta I \end{pmatrix}\right)^{-1} \begin{pmatrix} B \cr B \end{pmatrix}.\end{equation}
    Moreover, if $f$ is a rational matrix function then the colligation matrix can be chosen to act on finite dimensional spaces.
\end{theorem}

\begin{proof} The only if part follows immediately
  as $f$ given in \eqref{ff} belongs to ${\mathcal S}_{{\mathbb D}^2}({\mathcal U}, {\mathcal Y})$ and clearly $f(z,\zeta)=f(\zeta,z)$.

  For the if direction, since $f\in {\mathcal S}_{{\mathbb D}^2}({\mathcal U}, {\mathcal Y})$, by \cite{MR1207393},
  we may write 
  \begin{equation}\label{f2} f(z,\zeta)= D + \begin{pmatrix} C_1 & C_2 \end{pmatrix} \begin{pmatrix} zI & 0 \cr 0 & \zeta I \end{pmatrix}\left( I- \begin{pmatrix} A_{11} & A_{12} \cr A_{21} & A_{22} \end{pmatrix} \begin{pmatrix} zI & 0 \cr 0 & \zeta I \end{pmatrix}\right)^{-1} \begin{pmatrix} B_1 \cr B_2 \end{pmatrix}\end{equation} for some contractive colligation matrix 
  $$\hat{M}=\begin{pmatrix} A_{11} & A_{12} & B_1\cr A_{21} & A_{22}  & B_2 \cr C_1 & C_2 & D \end{pmatrix} .$$
  Observing now that $$f(z,\zeta) = \frac12 (f(z,\zeta) + f(\zeta,z)), $$ it is straightforward to see that $f$ can be written as in \eqref{ff} with
  $$ A_1= \begin{pmatrix} A_{11} & \hspace{-0.2cm} 0 \cr 0 & \hspace{-0.2cm} A_{22} \end{pmatrix} , A_2 = \begin{pmatrix} 0 & \hspace{-0.2cm} A_{12} \cr A_{21} & \hspace{-0.2cm} 0 \end{pmatrix}, C=\frac{1}{\sqrt{2}} \begin{pmatrix} C_1 & \hspace{-0.2cm} C_2 \end{pmatrix} , B=\frac{1}{\sqrt{2}} \begin{pmatrix} B_1 \cr B_2 \end{pmatrix}.$$
  It remains to observe that 
  $$ M= V^* \begin{pmatrix} \hat{M} & 0 \cr 0 & \hat{M} \end{pmatrix} V,
  $$ with $$ V=\begin{pmatrix} I & 0 & 0 & 0 & 0\cr 0 & 0 & 0 & I & 0 \cr 0 & 0 & 0 & 0 &\frac{1}{\sqrt{2}} I  \cr 0 & 0 & I & 0 & 0 \cr 0 & I & 0 & 0 & 0 \cr 0 & 0 & 0 & 0 & \frac{1}{\sqrt{2}} I  \end{pmatrix} , $$ so that $M$ is a contraction. Finally, the finite dimensionality in the rational matrix case follows from \cite{MR2839446} in the case of an inner function and from \cite[Theorem 1.3]{MR4359913} otherwise. 
\end{proof}

Next we consider determinantal representations of symmetric polynomials without roots in the (closed) bidisk. For polynomials in two variables without a symmetry requirement, this goes back to \cite{Kummert1989}; see also \cite{MR3441374}.

\begin{theorem}\label{detrepD2}
A scalar valued symmetric polynomial $p(z,\zeta)$ of degree $(n,n)$ has no roots in ${\mathbb D}^2$ if and only if $p$ has a determinantal representation
\begin{equation}\label{dr} p(z,\zeta ) = p(0,0) \det \left( I - \begin{pmatrix} A_1 & A_2\cr A_2 & A_1\end{pmatrix} \begin{pmatrix} zI & 0\cr 0 & \zeta I \end{pmatrix} \right)  \end{equation}
with 
\begin{equation}\label{contr} \begin{pmatrix} A_1 & A_2\cr A_2 & A_1\end{pmatrix} \end{equation}
a contraction. In addition, $p(z,\zeta)$ has no roots in $\overline{\mathbb D}^2$ if and only if $p$ has a determinantal representation \eqref{dr} with \eqref{contr} a strict contraction. In both cases the matrices $A_1$ and $A_2$ can be chosen to be of size at most $2n\times 2n$.  
\end{theorem}

We will provide the proof of this result in the next section. The results in \cite{Kummert1989} and \cite{MR3441374} suggest that perhaps one can choose $A_1$ and $A_2$ to be of size $n\times n$, but whether this indeed holds is an open question.
We note that \cite[Example 2.4]{MR3441374} provides an illustration of the above result.

Let us consider the following Nevanlinna-Pick interpolation problem with symmetric data:

\medskip

(NP${\mathbb D}^2$-symm) Given are $(z_i, \zeta_i) \in {\mathbb D}^2$ and $w_i \in {\mathbb C}$, $i=1,\ldots, n$. Find, if possible, $f\in {\mathcal S}_{{\mathbb D}^2}({\mathcal U}, {\mathcal Y})$ so that $f(z_i,\zeta_i)=w_i$ \underline{and} $f(\zeta_i,z_i)=w_i$ $i=1,\ldots , n$.

\begin{prop}\label{NP1}
    There is a solution to {\rm (NP${\mathbb D}^2$-symm)} if and only there is a symmetric solution to {\rm (NP${\mathbb D}^2$-symm)}.
\end{prop}

\begin{proof} The only if statement is trivial. For the if statement, observe that if
     $f$ solves (NP${\mathbb D}^2$-symm), then so does the symmetric function $\hat{f}(z,\zeta)=\frac12 (f(z,\zeta) + f(\zeta,z))$. 
\end{proof}

\section{The symmetrized bidisk}

As mentioned before, the symmetrized bidisk is the set \(\mathbb{G}=\{(s,p)\in \mathbb{C}^{2}:s=z+\zeta,p=z\zeta\text{\ for\ some\ }z,\zeta\in \mathbb{D}\}\).
We let ${\mathcal S}_{{\mathbb G}}({\mathcal U}, {\mathcal Y})$ denote the Schur class of analytic contractive ${\mathcal B}({\mathcal U}, {\mathcal Y})$-valued functions on ${\mathbb G}$. We start this section with a realization result, for which the infinite dimensional part was earlier established
in  \cite{MR3641771} and \cite{MR3724147}. Our proof relies on the results established in the previous section and also yields a finite dimensional version.

\begin{theorem}
    We have that $g(s,p)\in {\mathcal S}_{{\mathbb G}}({\mathcal U}, {\mathcal Y})$  if and only if there exists a contractive colligation matrix
    \begin{equation}\label{tM}\tilde{M}=\begin{pmatrix} \alpha_1 & 0 & \beta\cr 0 & \alpha_2 & 0 \cr \gamma & 0 & \delta \end{pmatrix} : \begin{matrix} {\mathcal H} \cr \oplus  \cr {\mathcal H} \cr \oplus\cr {\mathcal U}\end{matrix} \to \begin{matrix} {\mathcal H} \cr \oplus \cr  {\mathcal H}\cr \oplus \cr {\mathcal Y}\end{matrix}\end{equation} so that
    \begin{equation}\label{f3} g(s,p)= \delta + \frac12 \gamma (sI-2p\alpha_2)(I-\frac{s}{2} (\alpha_1+\alpha_2) +p\alpha_1\alpha_2)^{-1}\beta.\end{equation}
    Moreover, if $g$ is a rational matrix function then the colligation matrix can be chosen to act on finite dimensional spaces.
\end{theorem}

\begin{proof}
    Given $g(s,p)\in {\mathcal S}_{{\mathbb G}}({\mathcal U}, {\mathcal Y})$, we let $f(z,\zeta) = g(z+\zeta, z\zeta).$ Then $f$ satisfies the conditions of Theorem \ref{thm1} and thus we may write $f$ as in \eqref{ff}. Put now
    $$ \tilde{M}:= V^*MV,$$
    where \begin{equation}\label{U} V=U\oplus I , U=\begin{pmatrix} \frac{1}{\sqrt{2}} I & \frac{1}{\sqrt{2}} I \cr \frac{1}{\sqrt{2}} I & -\frac{1}{\sqrt{2}} I  \end{pmatrix}.\end{equation} Then $\tilde{M}$ equals \eqref{tM} where
    $$\alpha_1 = A_1+A_2, \alpha_2 = A_1-A_2, \gamma = \sqrt{2} C, \beta = \sqrt{2}B, \delta=D. $$ It is now a straightforward calculation that $g(s,p)$ equals \eqref{f3}. Indeed, assuming that $z\neq 0, \zeta \neq 0$, we may rewrite \eqref{ff} as
    \begin{equation}\label{f4} f(z,\zeta)= D + \begin{pmatrix} C &  \hspace{-0.2cm} C \end{pmatrix} U^* \left( U^*\begin{pmatrix} \frac{1}{z}I & 0 \cr 0 & \frac{1}{\zeta} \end{pmatrix} U - U^*\begin{pmatrix} A_1 & A_2 \cr A_2 & A_1 \end{pmatrix} U \right)^{-1} \hspace{-0.2cm} U\begin{pmatrix} B \cr B \end{pmatrix} =\end{equation}
    $$ \delta + \begin{pmatrix} \gamma & 0 \end{pmatrix} \left( \frac12\begin{pmatrix} \hat{s}I & \hat{d}I \cr \hat{d}I & \hat{s} I\end{pmatrix} - \begin{pmatrix} \alpha_1 & 0 \cr 0 & \alpha_2 \end{pmatrix} \right)^{-1}\begin{pmatrix} \beta \cr 0\end{pmatrix}= $$
    $$ \delta+ \gamma ( \frac12 \hat{s} - \alpha_1 -\frac14 \hat{d}^2(\frac12\hat{s} - \alpha_2)^{-1})^{-1} \beta, $$
    where $\hat{s}=\frac1z+\frac{1}{\zeta}, \hat{d} = \frac1z-\frac{1}{\zeta}.$ Let us also put $s=z+\zeta, p=z\zeta$, and observe that $\hat{s}=\frac{s}{p}, \hat{d}^2=\frac{s^2-4p}{p^2}$. Using these observations, it is a simple calculation to see that $g(s,p)=f(z_1,z_2)$ is as in \eqref{f3}. The latter equation is also satisfied when $z=0$ or $\zeta=0$ (using continuity).

    The steps in the above calculation can be reversed, so that if $g(s,p)$ is given by \eqref{f4}, then $g(s,p)=f(z,\zeta)$ with $f$ as in \eqref{ff}. But then, since $f \in {\mathcal S}_{{\mathbb D}^2}({\mathcal U}, {\mathcal Y})$, it follows that $g\in {\mathcal S}_{{\mathbb G}}({\mathcal U}, {\mathcal Y})$.
\end{proof}

\begin{remark}\label{remark}\rm
    Notice that the above results provide a way to find a realization for $g(s,p)$ as follows. Put $f(z,\zeta)=g(z+\zeta, z\zeta)$ and find a realization for $f$ as in \eqref{f2}, for instance by using the methods described in \cite{MR4359913}. Then setting 
    $$ \alpha_1 = \begin{pmatrix} A_{11} & A_{12}\cr A_{21} & A_{22}\end{pmatrix}, \alpha_2 = \begin{pmatrix} A_{11} & \hspace{-.2cm} -A_{12}\cr -A_{21} & \hspace{-.2cm} A_{22}\end{pmatrix}, \beta= \begin{pmatrix} B_1 \cr B_2 \end{pmatrix}, \gamma= \begin{pmatrix} C_1 & \hspace{-0.2cm} C_2 \end{pmatrix}, \delta=D, $$
    gives a realization \eqref{f3} for $g(s,p).$ 
\end{remark}

Next, we would like to present a determinantal representation result. 

\begin{theorem}\label{detrepG}
     Let $g(s,p)$ be a polynomial. Then $g(s,p)$ is without roots in ${\mathbb G}$ if and only if there  exists a contraction \eqref{contr} such that 
     % \begin{equation}\label{contract} \begin{pmatrix} A_1 & A_2 \cr A_2 & A_1 \end{pmatrix} \end{equation} so that
    \begin{equation}\label{gh} g(s,p)= g(0,0) \det (I -sA_1 + p(A_1+A_2)(A_1-A_2)).\end{equation}
    In addition, $g(s,p)$ is without roots in $\overline{\mathbb G}$ if and only if there exists a strict contraction \eqref{contr} so that 
    \eqref{gh} holds.
\end{theorem}

\begin{proof}
    Let $g(s,p)$ be without roots in ${\mathbb G}$. For ease of the presentation, let us assume that $g(0,0)=1$ (otherwise, consider $\frac{g(s,p)}{g(0,0)}$). Put $p(z,\zeta)=g(z+\zeta, z\zeta).$ Then $p(z,\zeta)$ is symmetric and without roots in the closed bidisk. Let $p(z,\zeta)$ be of degree $(n,n)$. Let $\sigma=\frac12(z+\zeta)$ and $\delta=\frac12 (z-\zeta)$. Then $p(z,\zeta)=p(\sigma+\delta,\sigma-\delta)=q(\sigma,\delta)$ for some polynomial $q$. Since $p(z,\zeta)=p(\zeta,z)$ we get that $q(\sigma,\delta)=q(\sigma,-\delta)$. Thus $q$ is even in $\delta$, so $q$ is a polynomial in $\sigma$ and $\delta^2$. Moreover, $q$ as a polynomial in $\sigma$ and $\delta^2$ is of degree $(n,m)$, where $m=\lfloor \frac{n}{2} \rfloor$, and $q$ is without roots in ${\mathbb D}^2$. Applying \cite[Theorem 2.1]{MR3441374} to $q$ as a polynomial in $\sigma$ and $\delta^2$, we obtain that there exists a $(n+m)\times (n+m)$ contraction
    $$ \begin{pmatrix} K_{11} & K_{12}\cr K_{21} & K_{22}\end{pmatrix}$$ such that
    $$ q(\sigma,\delta)= \det \left( I_{n+m} - 
    \begin{pmatrix} K_{11} & K_{12}\cr K_{21} & K_{22}\end{pmatrix} \begin{pmatrix} \sigma I_n & 0\cr 0 & \delta^2 I_m\end{pmatrix} \right).$$
    Let $s=z+\zeta$ and $p=z\zeta$. Then $\sigma=\frac{s}{2}$ and $\delta^2=-p+\frac{s^2}{4}$ and since $g(s,p)=q(\sigma,\delta)$, we find that
    $$ g(s,p) = \det \left( I - \begin{pmatrix} \frac12 K_{11} & -K_{12}\cr \frac12 K_{21} & -K_{22}\end{pmatrix} \begin{pmatrix} s I_n & 0\cr 0 & p I_m\end{pmatrix} - \frac{s^2}{4} \begin{pmatrix} 0 & K_{12}\cr 0 & K_{22}\end{pmatrix} \right). $$ We may rewrite this, using Schur complements, as
    $$ g(s,p) = \det \left( I - s \begin{pmatrix} \frac12 K_{11} & 0 & -\frac12 K_{12}\cr \frac12 K_{21} & 0 &  -\frac 12 K_{22}\cr 0 & -\frac12 I_m & 0 \end{pmatrix} +p \begin{pmatrix} 0 & K_{12} & 0\cr 0 & K_{22} & 0 \cr 0 & 0 & 0 \end{pmatrix} \right). $$
Let now $$ A_1= \begin{pmatrix} \frac12 K_{11} & 0 & -\frac12 K_{12}\cr \frac12 K_{21} & 0 &  -\frac 12 K_{22}\cr 0 & -\frac12 I_m & 0 \end{pmatrix} , A_2 = \begin{pmatrix} \frac12 K_{11} & 0 & -\frac12 K_{12}\cr \frac12 K_{21} & 0 &  -\frac 12 K_{22}\cr 0 & \frac12 I_m & 0 \end{pmatrix}.$$ Then \eqref{gh} follows. Note also that $A_1$ and $A_2$ are of size $n+2m \le 2n$. Finally, it is easy to check that $\eqref{contr}$ is a contraction (as, by using \eqref{U}, $\eqref{contr}$ being a contraction is equivalent to $\| A_1 \pm A_2 \| \le 1$). 

If $g(s,p)$ is without roots in $\overline{\mathbb G}$, then there exists $R>1$ so that $g(Rs,R^2p)$ is without roots $(s,p)\in {\mathbb G}$. This yields that $g(Rs,R^2p)$ may be expressed as the righthand side of \eqref{gh}. But then $\frac{1}{R}A_1$ and $\frac{1}{R} A_2$ yields the desired strict contraction in the determinantal representation \eqref{gh} for $g(s,p)$. 
\end{proof}

{\em Proof of Theorem \ref{detrepD2}.}
Let $p(z,\zeta)$ be a symmetric polynomial without roots in ${\mathbb D}^2$.
Let $s=z+\zeta$ and $d=z-\zeta$. Then $p(z,\zeta)=p(\frac12(s+d),\frac12 (s-d))=q(s,d)$ for some polynomial $q$. Since $p(z,\zeta)=p(\zeta,z)$ we get that $q(s,d)=q(s,-d)$. Thus $q$ is even in $d$, so $q$ is a polynomial in $s$ and $d^2$. As $d^2= s^2 -4p$, where $p=z\zeta$, we get that $q$ is in fact a polynomial $g(s,p)$ in $s$ and $p$. Thus $p(z,\zeta) = g(s,p)$, and as $p$ is without roots in ${\mathbb D}^2$ the polynomial $g(s,p)$ is without roots in ${\mathbb G}$. Now, Theorem \ref{detrepG} yields the existence of matrices $A_1$ and $A_2$ of size $\ell \le 2n$ so that \eqref{contr} is a contraction and \eqref{gh} holds. We claim that the right hand side of \eqref{dr} equals \eqref{gh}, and we use a similar calculation as in \eqref{f4}. For convenience we assume that $g(0,0)=p(0,0)=1$.

Let $z\neq 0 \neq \zeta$ and put $\hat{s}=\frac1z+\frac{1}{\zeta}, \hat{d} = \frac1z-\frac{1}{\zeta}.$ Then $\hat{s}=\frac{s}{p}, \hat{d}^2=\frac{s^2-4p}{p^2}$. Moreover, using $U$ as in \eqref{U},
$$ \det U\left( I - \begin{pmatrix} A_1 & A_2\cr A_2 & A_1\end{pmatrix} \begin{pmatrix} zI_\ell & 0\cr 0 & \zeta I_\ell \end{pmatrix} \right)U^* = $$ $$z^\ell \zeta^\ell \det \left( \frac12\begin{pmatrix} \hat{s}I & \hat{d}I \cr \hat{d}I & \hat{s} I\end{pmatrix} - \begin{pmatrix} A_1+A_2 & 0 \cr 0 & A_1-A_2 \end{pmatrix} \right) = $$
$$ z^\ell \zeta^\ell \det \left( (\frac{\hat{s}}{2}I_\ell -(A_1+A_2)) (\frac{\hat{s}}{2} I_\ell -(A_1-A_2)) - \frac{\hat{d}^2}{4}I_\ell \right) = $$
$$ p^\ell \det(\frac{1}{p}I_l - \frac{s}{p} A_1 +(A_1+A_2)(A_1-A_2)) = g(s,p).$$ This proves that \eqref{dr} holds (where we observe that for $z=0$ or $\zeta=0$ the equality follows by continuity).

When $p(z,\zeta)$ has no roots in $\overline{\mathbb D}^2$ one may find an $R>1$ so that $p(Rz,R\zeta)$ has not roots for $(z,\zeta) \in {\mathbb D}^2$. One can thus represent $p(Rz,R\zeta)$ as the righthand side of \eqref{dr}. But then {\tiny $\frac{1}{R} \begin{pmatrix} A_1 & A_2\cr A_2 & A_1\end{pmatrix}$} is the desired strict contraction.
\hfill $\Box$

\medskip

Let us consider the following Nevanlinna-Pick interpolation problem:

\medskip

(NP${\mathbb G}$) Given are $(s_i, p_i) \in {\mathbb G}$ and $w_i \in {\mathbb C}$, $i=1,\ldots, n$. Find, if possible, $g\in {\mathcal S}_{{\mathbb G}}({\mathcal U}, {\mathcal Y})$ so that $g(s_i,p_i)=w_i$, $i=1,\ldots , n$.

\medskip

We also introduce the associated Nevanlinna-Pick interpolation problem on ${\mathbb D}^2$. For each $(s_i, p_i) \in {\mathbb G}$ there is a pair $(z_i, \zeta_i) \in {\mathbb D}^2$ so that $s_i=z_i+\zeta_i$ and $p_i=z_i \zeta_i$. With this data we can now consider (NP${\mathbb D}_2$-symm). The following Proposition was observed earlier; see \cite{MR3641771, MR3724147, MR4900535}. %The pair $(z_{i+n}, \zeta_{i+n})=(\zeta_i, z_i)$ is the only that also has this property. We put $w_{i+n}=w_i$, $i=1,\ldots,n$.

\begin{prop}\label{NP}
    There is a solution to {\rm (NP${\mathbb G}$)} if and only there is a solution to the associated {\rm (NP${\mathbb D}^2$-symm)}.
\end{prop}

\begin{proof}
    Suppose $g\in {\mathcal S}_{{\mathbb G}}({\mathcal U}, {\mathcal Y})$ solves (NP${\mathbb G}$). Put $f(z,\zeta) = g(z+\zeta, z\zeta)$. Then $f$ solves (NP${\mathbb D}^2$-symm).

    Suppose that $f$ solves (NP${\mathbb D}^2$-symm). Then so does $\hat{f}(z,\zeta)=\frac12 (f(z,\zeta) + f(\zeta,z))$. The latter has a representation as in \eqref{ff}. Put $$\alpha_1 = A_1+A_2, \alpha_2 = A_1-A_2, \gamma = \sqrt{2} C, \beta = \sqrt{2}B, \delta=D ,  $$ and define $g$ via \eqref{f4}. Then $g\in {\mathcal S}_{{\mathbb G}}({\mathcal U}, {\mathcal Y})$ solves (NP${\mathbb G}$).
\end{proof}

While the papers \cite{MR3641771, MR3724147} provide necessary and sufficient conditions for a solution to (NP${\mathbb G}$), it may be easier to pull the problem back to the bidisk as suggested in \cite[before Theorem 5.1]{MR3641771}. Indeed, for (NP${\mathbb D}^2$-symm) the criterion comes down to checking the existence of a pair $(K_1,K_2)$ of positive semidefinite matrices satisfying the following linear constraints (see \cite{MR1665697})
$$ (1-\overline{w_i}w_j)_{i,j=1}^n = K_1 \circ (1-\overline{z_i}z_j)_{i,j=1}^n + K_2 \circ (1-\overline{\zeta_i}\zeta_j)_{i,j=1}^n, $$
and is thus verifiable by SemiDefinite Programming (SDP). Here $\circ$ denotes the Schur (=entrywise) product of matrices. In addition, algorithms have been developed to find solutions to (NP${\mathbb D}^2$-symm); see, e.g., \cite{MR1665697,AM, AMbook}. Thus, if we combine this with Remark \ref{remark}, we provide a way to find solutions to (NP${\mathbb G}$). We illustrate this with an example.

\begin{example}\rm Following \cite{MR4900535} we consider (NP${\mathbb G}$) with data
$$ s_1=0, p_1=0, w_1=0, s_2=0, p_2=\frac12, w_2=\frac12. $$ Then the associated (NP${\mathbb D}^2$-symm) problem is
$$ z_1=0, \zeta_1=0, w_1=0, z_2=\frac{i}{\sqrt{2}}=\zeta_3, \zeta_2=\frac{-i}{\sqrt{2}}=z_3, w_2=\frac12=w_3. $$
Checking the conditions for a solution for (NP${\mathbb D}^2$), we need to find positive semidefinite matrices $K_1$ and $K_2$ so that
$$ \begin{pmatrix}
    1 & 1 & 1\cr 1 & \frac34 & \frac34 \cr 1  & \frac34 & \frac34 
\end{pmatrix}= K_1 \circ \begin{pmatrix}
    1 & 1 & 1\cr 1 & \frac12 & \frac32 \cr 1 & \frac32 & \frac12 
\end{pmatrix} + K_2 \circ \begin{pmatrix}
    1 & 1 & 1\cr 1 & \frac12 & \frac32 \cr 1 & \frac32 & \frac12 
\end{pmatrix}. $$
 One easily finds the choice
$$ K_1=K_2 = \begin{pmatrix}
    \frac12 & \frac12 & \frac12 \cr \frac12 & \frac34 & \frac14 \cr \frac12 & \frac14 & \frac34 
\end{pmatrix},$$
which is positive semidefinite. To obtain a function that interpolates the given data, one needs to perform a lurking isometry (or, actually, lurking contraction) argument. Note that we may factor $K_1=K_2=U^*U$ with 
$$U = \begin{pmatrix}
    \frac{1}{\sqrt{2}} & \frac{1}{\sqrt{2}} & \frac{1}{\sqrt{2}} \cr 0 & \frac{1}{2}  &-\frac{1}{2}
\end{pmatrix}.$$
Hence if  $$u_1 = \begin{pmatrix}
    \frac{1}{\sqrt{2}} \cr 0
\end{pmatrix}, \quad u_2 = \begin{pmatrix}
    \frac{1}{\sqrt{2}} \cr \frac{1}{2}
\end{pmatrix}, \quad u_3 = \begin{pmatrix}
    \frac{1}{\sqrt{2}} \cr -\frac{1}{2}
\end{pmatrix},$$
we are looking for a contraction 
$$V: \left \langle \begin{pmatrix}
    1 \cr z_ju_j \cr \zeta_ju_j
\end{pmatrix}\right\rangle \to \left \langle \begin{pmatrix} w_j \cr u_j \cr u_j
\end{pmatrix} \right \rangle, \quad j=1,2,3$$
acting on $\mathbb{C}^5$. This leads to
$$V
\begin{pmatrix}
1 & 1 & 1 \cr
0 & \frac{i}{2} & -\frac{i}{2} \cr
0 & \frac{i}{2\sqrt{2}} & \frac{i}{2\sqrt{2}} \cr
0 & -\frac{i}{2} & \frac{i}{2} \cr
0 & -\frac{i}{2\sqrt{2}} & -\frac{i}{2\sqrt{2}}
\end{pmatrix}
=
\begin{pmatrix}
0 & \frac{1}{2} & \frac{1}{2} \cr
\frac{1}{\sqrt{2}} & \frac{1}{\sqrt{2}} & \frac{1}{\sqrt{2}} \cr
0 & \frac{1}{2} & -\frac{1}{2} \cr
\frac{1}{\sqrt{2}} & \frac{1}{\sqrt{2}} & \frac{1}{\sqrt{2}} \cr
0 & \frac{1}{2} & -\frac{1}{2}
\end{pmatrix}.$$
One of the possible choices is
$$V = \begin{pmatrix}
0 & 0 & -\frac{i}{\sqrt{2}} & 0 & \frac{i}{\sqrt{2}} \cr
\frac{1}{\sqrt{2}} & 0 & 0 & 0 & 0 \cr
0 & 0 & 0 & i & 0 \cr
\frac{1}{\sqrt{2}} & 0 & 0 & 0 & 0 \cr
0 & -i & 0 & 0 & 0
\end{pmatrix}.$$
Using Remark \ref{remark}, this leads to 
$$\alpha_1 = \begin{pmatrix} 
0 & 0 & 0 & 0 \cr 0 & 0 & i & 0 \cr0 & 0 & 0 & 0 \cr-i & 0 & 0 & 0
\end{pmatrix}, \alpha_2 = \begin{pmatrix} 
0 & 0 & 0 & 0 \cr 0 & 0 & -i & 0 \cr0 & 0 & 0 & 0 \cr i & 0 & 0 & 0
\end{pmatrix}, $$ $$ \beta = \begin{pmatrix} \frac{1}{\sqrt{2}} \cr 0 \cr \frac{1}{\sqrt{2}} \cr 0\end{pmatrix}, \gamma = \begin{pmatrix}
    0 & -\frac{i}{\sqrt{2}} & 0 & \frac{i}{\sqrt{2}}
\end{pmatrix}, \delta = 0,$$ 
and formula \eqref{f3} gives $g(s,p)=p$ as a solution to (NP${\mathbb G}$). 

Another choice for $V$ is
$$\begin{pmatrix}
0 & 0 & -\frac{i}{\sqrt{2}} & 0 & \frac{i}{\sqrt{2}} \cr
\frac{1}{\sqrt{2}} & 0 & 0 & 0 & 0 \cr
0 & -i & 0 & 0 & 0 \cr
\frac{1}{\sqrt{2}} & 0 & 0 & 0 & 0 \cr
0 & 0 & 0 & i & 0
\end{pmatrix},$$ which leads to the solution $g(s,p)=p-\frac{s^2}{2}$.
    
\end{example}

Note that in \cite[Theorem 3.5]{MR4900535} it is shown that if (NP${\mathbb G}$) has a solution, it in fact has a rational inner solution ($g(s,p)=p$ in the above example is inner). 

\subsection*{Acknowledgment} R.
Baran gratefully acknowledges support from the ID.UJ ``Initiative of Excellence - Research University".

\bigskip 
\noindent {\bf Data availability statement.} Data sharing is not applicable to this article as no
datasets were generated or analyzed during the current study.

\bigskip 
\noindent{\bf Conflicts of interest statement.} The authors have no conflicts of interest to declare that are relevant to the
content of this article.

\bibliographystyle{plain}
    \bibliography{refs}

\end{document}